\documentclass[10pt]{amsart}
\usepackage{hyperref, amssymb}
\usepackage{color}
\usepackage[all]{xy}

\theoremstyle{plain}
\newtheorem{theorem}{Theorem}[section]
\newtheorem{proposition}[theorem]{Proposition}
\newtheorem{lemma}[theorem]{Lemma}
\newtheorem{corollary}[theorem]{Corollary}

\theoremstyle{definition}

\newtheorem{remark}[theorem]{Remark}

\newcommand{\f}{\varphi}

\newcommand{\CC}{\mathbb C}
\newcommand{\PP}{\mathbb P}
\newcommand{\ZZ}{\mathbb Z}
\newcommand{\MM}{\mathbf M}

\newcommand{\E}{{\mathcal E}}
\newcommand{\F}{{\mathcal F}}
\newcommand{\G}{{\mathcal G}}
\newcommand{\I}{{\mathcal I}}

\def\O{\mathcal O}

\newcommand{\Ker}{{\mathcal Ker}}
\newcommand{\Coker}{{\mathcal Coker}}
\newcommand{\Image}{{\mathcal Im}}

\newcommand{\Ext}{\operatorname{Ext}}
\newcommand{\Hilb}{\operatorname{Hilb}}
\newcommand{\Hom}{\operatorname{Hom}}
\def\H{\operatorname{H}}
\newcommand{\M}{\operatorname{M}}

\newcommand{\Poly}{\operatorname{P}}
\newcommand{\p}{\operatorname{p}}

\newcommand{\Sym}{\operatorname{S}}
\newcommand{\Syst}{\operatorname{Syst}}

\newcommand{\trans}{{\scriptscriptstyle \operatorname{T}}}

\newcommand{\tensor}{\otimes}

\newcommand{\lra}{\longrightarrow}

\begin{document}

\title[Moduli of stable sheaves supported on curves of genus $3$ in $\PP^1 \times \PP^1$]
{Moduli of stable sheaves supported on curves of genus three contained in a quadric surface}

\author{Mario Maican}
\address{Institute of Mathematics of the Romanian Academy, Calea Grivitei 21, Bucharest 010702, Romania}

\email{maican@imar.ro}

\begin{abstract}
We study the moduli space of stable sheaves of Euler characteristic $1$
supported on curves of arithmetic genus $3$ contained in a smooth quadric surface.
We show that this moduli space is rational.
We compute its Betti numbers by studying the variation of the moduli spaces of $\alpha$-semi-stable pairs.
We classify the stable sheaves using locally free resolutions or extensions.
We give a global description: the moduli space is obtained from a certain flag Hilbert scheme by performing two flips
followed by a blow-down.
\end{abstract}


\maketitle

\section{Introduction}
\label{introduction}

Let $\PP^1$ be the projective line over $\CC$ and consider the surface $\PP^1 \times \PP^1$ with fixed polarization
$\O(1, 1) = \O_{\PP^1}(1) \tensor \O_{\PP^1}(1)$.
For a coherent algebraic sheaf $\F$ on $\PP^1 \times \PP^1$, with support of dimension $1$, the Euler characteristic
$\chi(\F(m, n))$ is a polynomial expression in $m$, $n$, of the form
\[
P_{\F}(m, n) = rm + sn + t,
\]
where $r$, $s$, $t$ are integers depending only on $\F$. This is the \emph{Hilbert polynomial} of $\F$.
The \emph{slope} of $\F$ is
\[
\p(\F) = \frac{t}{r+s}.
\]
Let $\M(P)$ be the coarse moduli space of S-equivalence classes of sheaves on $\PP^1 \times \PP^1$
that are semi-stable with respect to the fixed polarization and that have Hilbert polynomial $P$.
We recall that $\F$ is semi-stable, respectively, stable, if it is pure and for any proper subsheaf $\F' \subset \F$ we have $\p(\F') \le \p(\F)$,
respectively, $\p(\F') < \p(\F)$.
According to \cite{lepotier}, $\M(P)$ is projective, irreducible, and smooth at points given  by stable sheaves.
Its dimension is $2rs + 1$ if $r > 0$ and $s > 0$.
The spaces $\M(rm + n + 1)$, $\M(2m + 2n + 1)$ and $\M(2m + 2n + 2)$ were studied in \cite{ballico_huh}.
In fact, it is not difficult to see that $\M(rm + n + 1)$ consists of the structure sheaves of curves of degree $(1, r)$,
so it is isomorphic to $\PP^{2r + 1}$.
The space $\M(3m + 2n + 1)$ was studied in \cite{choi_katz_klemm} and \cite{genus_two}.
We refer to the introductory section of \cite{genus_two} for more background information.

This paper is concerned with the study of $\MM = \M(4m + 2n + 1)$.
The closed points of $\MM$ are in a bijective correspondence with the isomorphism classes $[\F]$ of stable sheaves $\F$
supported on curves of degree $(2, 4)$ and satisfying the condition $\chi(\F) = 1$.
As already mentioned, $\MM$ is a smooth irreducible projective variety of dimension $17$.
For any $t \in \ZZ$, twisting by $\O(t,t)$ gives an isomorphism $\MM \simeq \M(4m + 2n + 6t + 1)$.
According to \cite[Corollary 1]{genus_two}, $\MM \simeq \M(4m + 2n -1)$.
In the following theorem we classify the sheaves in $\MM$.

\begin{theorem}
\label{main_theorem}
The variety $\MM$ can be decomposed into an open subset $\MM_0$, two closed irreducible subsets $\MM_2^{}$, $\MM_2'$,
each of codimension $2$, a locally closed irreducible subset $\MM_3$ of codimension $3$,
and a locally closed irreducible subset $\MM_4$ of codimension $4$.
These subsets are defined as follows: $\MM_0$ is the set of sheaves $\F$ having a resolution of the form
\[
0 \lra \O(-1, -3) \oplus \O(0, -3) \oplus \O(-1, -2) \stackrel{\f}{\lra} \O(0, -2) \oplus \O(0, -2) \oplus \O \lra \F \lra 0,
\]
where the entries $\f_{12}$ and $\f_{22}$ are linearly independent and the maximal minors of the matrix $(\f_{ij})_{i = 1, 2, j = 1, 2, 3}$,
describing the corestriction of $\f$ to the first two summands, have no common factor;
$\MM_2$ is the set of sheaves $\F$ having a resolution of the form
\[
0 \lra \O(-2, -2) \oplus \O(-1, -3) \overset{\f}{\lra} \O(-1, -2) \oplus \O(0, 1) \lra \F \lra 0,
\]
with $\f_{11} \neq 0$, $\f_{12} \neq 0$;
$\MM_2'$ is the set of sheaves $\F$ having a resolution of the form
\[
0 \lra \O(-2, -1) \oplus \O(-1, -4) \overset{\f}{\lra} \O(-1, -1) \oplus \O \lra \F \lra 0,
\]
with $\f_{11} \neq 0$, $\f_{12} \neq 0$;
$\MM_4$ is the set of extensions of the form
\[
0 \lra \O_Q \lra \F \lra \O_L(1, 0) \lra 0
\]
satisfying the condition $\H^0(\F) \simeq \CC$,
where $Q \subset \PP^1 \times \PP^1$ is a quintic curve of degree $(2, 3)$ and $L \subset \PP^1 \times \PP^1$ is a line of degree $(0, 1)$;
$\MM_3$ is the set of extensions of the form
\[
0 \lra \O_Q(p) \lra \F \lra \O_L \lra 0,
\]
where $\O_Q(p)$ is a non-split extension of $\CC_p$ by $\O_Q$ for a point $p \in Q$, and satisfying the condition $\H^0(\F) \simeq \CC$.

Moreover, $\MM_2$ is the Brill-Noether locus of sheaves for which $\H^1(\F) \neq \{ 0 \}$.
\end{theorem}

\noindent
The proof of Theorem \ref{main_theorem}, given in Section \ref{classification}, relies on the Beilinson spectral sequence,
which we recall in Section \ref{preliminaries}.
The varieties $X$ that appear in this paper have no odd homology, so we can define the Poincar\'e polynomial
\[
\Poly(X)(\xi) = \sum_{i \ge 0} \dim_{\mathbb Q}^{} \H^i(X, {\mathbb Q}) \xi^{i/2}.
\]

\begin{theorem}
\label{poincare_polynomial}
The Euler characteristic of $\MM$ is $288$.
The Poincar\'e polynomial of $\MM$ is
\begin{multline*}
\xi^{17} + 3\xi^{16} + 8\xi^{15} + 16\xi^{14} + 21\xi^{13} + 23\xi^{12} + 24\xi^{11} + 24\xi^{10} + 24\xi^9 \\
+ 24\xi^8 + 24\xi^7 + 24\xi^6 + 23\xi^5 + 21\xi^4 + 16\xi^3 + 8\xi^2 + 3\xi + 1.
\end{multline*}
\end{theorem}

\noindent
The proof of this theorem rests on the wall-crossing method of Choi and Chung \cite{choi_chung}.
In Section \ref{variation} we investigate how the moduli spaces $\M^{\alpha}(4m + 2n + 1)$ of $\alpha$-semi-stable pairs
with Hilbert polynomial $4m + 2n + 1$ change as the parameter $\alpha$ varies.
In Theorem \ref{wall_crossing} we find that $\M^{\alpha}(4m + 2n + 1)$ are related by two explicitly described flipping diagrams.
Combining this with Proposition \ref{blow_up}, yields a global description: $\MM$ is obtained from the flag Hilbert scheme
of three points on curves of degree $(2, 4)$ in $\PP^1 \times \PP^1$ by performing two flips followed by a blow-down
with center the Brill-Noether locus $\MM_2$.

The total space $X$ of $\omega_{\PP^1 \times \PP^1}$ is a Calabi-Yau threefold.
For a homology class $\beta = (r, s) \in \H_2(\PP^1 \times \PP^1) \subset \H_2(X)$ let
$N_{\beta}(X)$ be the genus zero Gromov-Witten invariant of $X$
and let $n_{\beta}(X)$ be the genus zero Gopakumar-Vafa invariant of $X$, as introduced in \cite{katz}.
It was noticed in \cite{choi_katz_klemm} that, up to sign, the latter is the Euler characteristic of a moduli space:
\[
n_{\beta}(X) = (-1)^{\dim \M(rm + sn + 1)} e(\M(rm + sn + 1)).
\]
In \cite{katz} Katz conjectured the relation
\[
N_{\beta}(X) = \sum_{k | \beta} \frac{n_{\beta/k}(X)}{k^3}.
\]
For $\beta = (4, 2)$, this conjecture reads
\begin{align*}
N_{(4, 2)}(X) = & (-1)^{\dim \M(4m + 2n + 1)} e(\M(4m + 2n + 1)) + \frac{1}{8}(-1)^{\dim \M(2m + n + 1)} e(\M(2m + n + 1)) \\
= & (-1)^{\dim \MM} e(\MM) + \frac{1}{8} (-1)^{\dim \PP^5} e(\PP^5) = (-1)^{17} 288 + \frac{1}{8} (-1)^5 6 = -288.75
\end{align*}

%%%%%%%%%%%%%%%%%%%%%%%%%%%%%%%%%%%%%%%%%%%%%%%%%%%%%%%%%%%%%
\section{Preliminaries}
\label{preliminaries}

Our main technical tool in Section \ref{classification} will be the Beilinson spectral sequence.
Let $\F$ be a coherent sheaf on $\PP^1 \times \PP^1$. According to \cite[Lemma 1]{buchdahl}, we have a spectral sequence
converging to $\F$, whose first level $E_1$ has display diagram
\begin{equation}
\label{E_1}
\xymatrix
{
\H^2(\F(-1, -1)) \tensor \O(-1, -1) = E_1^{-2, 2} \ar[r] & E_1^{-1, 2} \ar[r] & E_1^{0, 2} = \H^2(\F) \tensor \O \\
\H^1(\F(-1, -1)) \tensor \O(-1, -1) = E_1^{-2, 1} \ar[r]^-{\theta_1} & E_1^{-1, 1} \ar[r]^-{\theta_2} & E_1^{0, 1} = \H^1(\F) \tensor \O \\
\H^0(\F(-1, -1)) \tensor \O(-1, -1) = E_1^{-2, 0} \ar[r]^-{\theta_3} & E_1^{-1, 0} \ar[r]^-{\theta_4} & E_1^{0, 0} = \H^0(\F) \tensor \O
}
\end{equation}
where $E_1^{ij} = \{ 0 \}$ if $i \notin \{ -2, -1, 0 \}$ or if $j \notin \{ 0, 1, 2 \}$ and
\begin{equation}
\label{E_1^{-1,j}}
E_1^{-1, j} = \H^j(\F(0, -1)) \tensor \O(0, -1) \oplus \H^j(\F(-1, 0)) \tensor \O(-1, 0).
\end{equation}
If $\F$ has support of dimension $1$, then the first row of (\ref{E_1}) vanishes and
the convergence of the spectral sequence forces $\theta_2$ to be surjective and yields the exact sequence
\begin{equation}
\label{convergence}
0 \lra \Ker(\theta_1) \stackrel{\theta_5}{\lra} \Coker(\theta_4) \lra \F \lra \Ker(\theta_2)/\Image(\theta_1) \lra 0.
\end{equation}
An application of the Beilinson spectral sequence is the following lemma that will be used in Section \ref{classification}.

\begin{lemma}
\label{length_3_scheme}
Let $Z \subset \PP^1 \times \PP^1$ be a zero-dimensional subscheme of length $3$
that is not contained in a line of degree $(1, 0)$ or $(0, 1)$.
Then the ideal of $Z$ has resolution
\[
0 \lra 2\O(-2, -2) \stackrel{\zeta^\trans}{\lra} \O(-1, -2) \oplus \O(-2, -1) \oplus \O(-1, -1) \lra \I_Z \lra 0,
\]
where the maximal minors of $\zeta$ have no common factor.
The dual of the structure sheaf of $Z$ has resolution
\begin{equation}
\label{Z_resolution}
0 \lra \O(-2, -4) \lra \O(-1, -3) \oplus \O(0, -3) \oplus \O(-1, -2) \stackrel{\zeta}{\lra} 2\O(0, -2) \lra {\mathcal Ext}^2(\O_Z, \O) \lra 0.
\end{equation}
\end{lemma}

\begin{proof}
We apply the spectral sequence (\ref{E_1}) to the sheaf $\F = \I_Z(1, 1)$.
By hypothesis, $\H^0(\I_Z(1, 0)) = \{ 0 \}$ and $\H^0(\I_Z(0, 1)) = \{ 0 \}$ hence, from (\ref{E_1^{-1,j}}), we obtain the vanishing of $E_1^{-1, 0}$.
Since $\H^0(\I_Z) = \{ 0 \}$, also $E_1^{-2, 0}$ vanishes.
From the short exact sequence
\[
0 \lra \I_Z \lra \O \lra \O_Z \lra 0
\]
we obtain the vanishing of $\H^2(\I_Z)$. Analogously, $\H^2(\I_Z(1, 0))$, $\H^2(\I_Z(0, 1))$ and $\H^2(\I_Z(1, 1))$ vanish.
The first row of (\ref{E_1}) vanishes.
Denote $d = \dim_{\CC}^{} \H^1(\I_Z(1, 1))$. Display diagram (\ref{E_1}) now takes the simplified form
\[
\xymatrix
{
0 \ar[r] & 0 \ar[r] & 0 \\
2\O(-1, -1) \ar[r]^-{\theta_1} & \O(0, -1) \oplus \O(-1, 0) \ar[r]^-{\theta_2} & d\O \\
0 \ar[r] & 0 \ar[r] & (d + 1)\O
}
\]
From the convergence of the spectral sequence we see that $\theta_2$ is surjective.
There is no surjective morphism $\theta_2 \colon \O(0, -1) \oplus \O(-1, 0) \to d\O$ for $d \ge 1$, hence $d = 0$.
Thus, $\Ker(\theta_1)$ is a subsheaf of $\O$. We claim that $\Ker(\theta_1) = \{ 0 \}$.
Indeed, if $\Ker(\theta_1)$ were non-zero, then $\O/\Ker(\theta_1)$ would be a torsion subsheaf of $\I_Z(1, 1)$.
Combining the exact sequences
\[
0 \lra \O \lra \I_Z(1, 1) \lra \Coker(\theta_1) \lra 0,
\]
\[
0 \lra 2\O(-1, -1) \lra \O(0, -1) \oplus \O(-1, 0) \lra \Coker(\theta_1) \lra 0
\]
yields the resolution
\[
0 \lra 2\O(-1, -1) \lra \O(0, -1) \oplus \O(-1, 0) \oplus \O \lra \I_Z(1, 1) \lra 0.
\]
Applying ${\mathcal Hom}(-, \O(-1, -3))$, we obtain resolution (\ref{Z_resolution}).
If the maximal minors of the matrix representing $\zeta$ had a common factor $f$,
then the reduced support of $\Coker(\zeta)$ would contain the curve $\{ f = 0 \}$.
But this is impossible because ${\mathcal Ext}^2(\O_Z, \O)$ has support of dimension zero.
\end{proof}

\begin{lemma}
\label{unique_extension}
Let $S$ be a smooth projective surface and let $C \subset S$ be a locally Cohen-Macaulay curve.
Let ${\mathcal Z}$ be a coherent sheaf on $S$ with support of dimension zero.
Let $\F$ be an extension of ${\mathcal Z}$ by $\O_C$ without zero-dimensional torsion.
Then $\F$ is uniquely determined up to isomorphism,
meaning that if $\F'$ is another extension of ${\mathcal Z}$ by $\O_C$ without zero-dimensional torsion, then $\F' \simeq \F$.
Moreover, ${\mathcal Z} \simeq {\mathcal Ext}^2_{\O_S}(\O_Z, \O_S)$ for a subscheme $Z \subset C$ of dimension zero,
so we have the exact sequence
\begin{equation}
\label{C_F_Z_dual}
0 \lra \O_C \lra \F \lra {\mathcal Ext}^2_{\O_S}(\O_Z, \O_S) \lra 0.
\end{equation}
\end{lemma}

\begin{proof}
This lemma is a direct consequence of \cite[Proposition B.5]{pandharipande_thomas}.
Indeed, given an exact sequence
\begin{equation}
\label{C_F_C}
0 \lra \O_C \lra \F \lra {\mathcal Z} \lra 0
\end{equation}
in which $\F$ has no zero-dimensional torsion, then the pair $(\O_C, \F)$ is a stable pair supported on $C$, in the sense of \cite{pandharipande_thomas}.
By \cite[Lemma B.2]{pandharipande_thomas}, we have ${\mathcal Ext}_{\O_C}^1(\F, \O_C) = \{ 0 \}$.
Applying ${\mathcal Hom}_{\O_C}^{}(-, \O_C)$ to (\ref{C_F_C}), yields the exact sequence
\begin{equation}
\label{C_F_C_dual}
0 \lra {\mathcal Hom}^{}_{\O_C} (\F, \O_C) \lra \O_C \lra {\mathcal Ext}^1_{\O_C}({\mathcal Z}, \O_C) \lra 0.
\end{equation}
Thus, ${\mathcal Ext}^1_{\O_C}({\mathcal Z}, \O_C)$ is the structure sheaf $\O_Z$ of a zero-dimensional subscheme $Z \subset C$.
Under the bijection of \cite[Proposition B.5]{pandharipande_thomas} between stable pairs supported on $C$
and zero-dimensional subschemes of $C$, the pair $(\O_C, \F)$ corresponds to $Z$,
so it is uniquely determined, up to isomorphism.
Tensoring (\ref{C_F_C_dual}) with the dualising line bundle $\omega_C$ on $C$, yields the exact sequence
\begin{equation}
\label{F_dual_C_Z}
0 \lra {\mathcal Hom}(\F, \omega_C) \lra \omega_C \lra \O_Z \lra 0.
\end{equation}
We claim that ${\mathcal Hom}(\F, \omega_C) \simeq {\mathcal Ext}^1(\F, \omega_S)$.
This follows by applying ${\mathcal Hom}(\F, -)$ to the exact sequence
\[
0 \lra \omega_S \lra \omega_S \tensor \O(C) \lra \omega_S \tensor \O(C)|_{C} \simeq \omega_C \lra 0.
\]
We obtain the exact sequence
\[
0 \lra {\mathcal Hom}(\F, \omega_C) \lra {\mathcal Ext}^1(\F, \omega_S) \lra {\mathcal Ext}^1(\F, \omega_S \tensor \O(C)).
\]
The last morphism is locally multiplication with an equation $f$ defining $C$.
But $C = \operatorname{supp}(\F)$, hence $f$ annihilates $\F$,
and hence $f$ annihilates ${\mathcal Ext}^1(\F, \omega_S)$. This proves the claim.
According to \cite[Remark 4]{rendiconti}, ${\mathcal Ext}^1({\mathcal Ext^1}(\F, \omega_S), \omega_S) \simeq \F$.
Clearly,
\[
{\mathcal Ext}^1(\O_Z, \omega_S) = \{ 0 \}, \qquad {\mathcal Ext}^1(\omega_C, \omega_S) \simeq \O_C, \qquad {\mathcal Ext}^2(\omega_C, \omega_S) = \{ 0 \}.
\]
Applying ${\mathcal Hom}(-, \omega_S)$ to (\ref{F_dual_C_Z}) yields extension (\ref{C_F_Z_dual}).
Comparing with (\ref{C_F_C}), we see that ${\mathcal Z} \simeq {\mathcal Ext}^2(\O_Z, \O_S)$.
\end{proof}

\noindent
Crucial for our classification of semi-stable sheaves
is the following vanishing result that should be compared with \cite[Proposition 4]{genus_two}.
We fix vector spaces $V_1$ and $V_2$ over $\CC$ of dimension $2$ and we identify $\PP^1 \times \PP^1$
with $\PP(V_1) \times \PP(V_2)$. Let $\{ x, y \}$ be a basis of $V_1^*$ and let $\{z, w \}$ be a basis of $V_2^*$.
A morphism $\O(i, j) \to \O(k, l)$ will be represented by a form in $\Sym^{k - i} V_1^* \tensor \Sym^{l - j} V_2^*$.

\begin{proposition}
\label{vanishing}
Assume that the sheaf $\F$ gives a point in $\MM$.
\begin{enumerate}
\item[(i)] We have $\H^0(\F(-1, -1)) = \{ 0 \}$ and $\H^0(\F(-1, 0)) = \{ 0 \}$.
\item[(ii)] If $\F$ satisfies the vanishing condition $\H^0(\F(0, -1)) = \{ 0 \}$, then $\H^1(\F) = \{ 0 \}$.
\end{enumerate}
\end{proposition}

\begin{proof}
(i) The vanishing of $\H^0(\F(-1, -1))$ follows from \cite[Proposition 2(i)]{genus_two}.
To prove the vanishing of $\H^0(\F(-1, 0))$ we can argue as in the proof of \cite[Proposition 3]{genus_two}.

\medskip

\noindent
(ii) Assume now that $\H^0(\F(0, -1)) = \{ 0 \}$. 
From (\ref{E_1^{-1,j}}) and part (i) of the proposition, we deduce that $E_1^{-1, 1} \simeq \O(0, -1) \oplus 3\O(-1, 0)$.
Denote $d = \dim_{\CC}^{} \H^1(\F)$.
There is no surjective morphism
\[
\theta_2 \colon \O(0, -1) \oplus 3\O(-1, 0) \lra d\O
\]
for $d \ge 4$, hence $d \le 3$. Assume that $d = 3$. The maximal minors for a matrix representation of $\theta_2$ have no common factor,
otherwise $\theta_2$ would not be surjective.
Thus, $\Ker(\theta_2) \simeq \O(-3, -1)$, hence $\theta_1 = 0$, and hence,
from the exact sequence (\ref{convergence}), we obtain a surjective morphism
$\F \to \O(-3, -1)$. This is absurd. Thus, the case when $d = 3$ is unfeasible.

Consider now the case when $d = 2$. If $\theta_2$ is represented by a matrix of the form
\[
A = \left[
\begin{array}{cccc}
0 & \star & \star & \star \\
0 & \star & \star & \star
\end{array}
\right],
\]
then $\Ker(\theta_2) \simeq \O(0, -1) \oplus \O(-3, 0)$, hence $\O(-3, 0)$ is a direct summand of $\Ker(\theta_2)/\Image(\theta_1)$,
and hence, from the exact sequence (\ref{convergence}), we obtain a surjective morphism $\F \to \O(-3, 0)$. This is absurd.
If $\theta_2$ is represented by a matrix of the form
\[
B = \left[
\begin{array}{cccc}
\star & \star & \star & 0 \\
\star & \star & \star & 0
\end{array}
\right],
\]
then $\Ker(\theta_2) \simeq \O(-2, -1) \oplus \O(-1, 0)$, hence $\O(-2, -1)$ is a direct summand of $\Ker(\theta_2)/\Image(\theta_1)$,
and hence we obtain a surjective morphism $\F \to \O(-2, -1)$. This is absurd.
If $\theta_2$ is represented by a matrix of the form
\[
C = \left[
\begin{array}{cccc}
1 \tensor u & v \tensor 1 & 0 & 0 \\
0 & 0 & x \tensor 1 & y \tensor 1
\end{array}
\right],
\]
then $\Ker(\theta_2) \simeq \O(-1, -1) \oplus \O(-2, 0)$ and we obtain a surjective morphism $\F \to \O(-2, 0)$. This is absurd.
We claim that, if $\theta_2$ is not of the form $A$, $B$ or $C$, then $\theta_2$ is represented by a matrix of the form
\[
D = \left[
\begin{array}{cccc}
- 1 \tensor z & x \tensor 1 & y \tensor 1 & 0 \\
\star & \star & \star & v \tensor 1
\end{array}
\right],
\]
with $v \neq 0$. Indeed, since $\theta_2 \nsim A$ and $\theta_2 \nsim B$, we may write
\[
\theta_2 = \left[
\begin{array}{cccc}
1 \tensor u & v_1 \tensor 1 & v_2 \tensor 1 & 0 \\
\star & \star & \star & v \tensor 1
\end{array}
\right]
\]
with $u \neq 0$, $v \neq 0$. Since $\theta_2 \nsim B$, $v_1$ and $v_2$ cannot be both zero.
If $v_1$ and $v_2$ are linearly independent, then $\theta_2 \sim D$.
If $v_1$ and $v_2$ span a one-dimensional vector space,
then, since $\theta_2 \nsim B$, we may write
\[
\theta_2 = \left[
\begin{array}{cccc}
1 \tensor u\phantom{_1} & v_1 \tensor 1 & 0 & 0 \\
1 \tensor u_1 & 0 & x \tensor 1 & y \tensor 1
\end{array}
\right].
\]
Since $\theta_2 \nsim C$, we have $u_1 \neq 0$, forcing $\theta_2 \sim D$.
In the case when $\theta_2 = D$, it is easy to see that the morphism
\[
\theta_1 \colon 5\O(-1, -1) \lra \O(0, -1) \oplus 3\O(-1, 0)
\]
is represented by a matrix of the form
\[
\left[
\begin{array}{ccccc}
x \tensor 1 & y \tensor 1 & 0 & 0 & 0 \\
1 \tensor z & 0 & 0 & 0 & 0 \\
0 & 1 \tensor z & 0 & 0 & 0 \\
\star & \star & 0 & 0 & 0
\end{array}
\right],
\]
hence $\Ker(\theta_1) \simeq 3\O(-1, -1)$, and hence $\Coker(\theta_5)$ has Hilbert polynomial $3m + 3n + 3$.
But then, in view of the exact sequence (\ref{convergence}),
$\Coker(\theta_5)$ is a destabilizing subsheaf of $\F$. Thus, the case when $d = 2$ is also unfeasible.

It remains to examine the case when $d = 1$. Recall that $\theta_2$ is surjective, hence it can have two possible forms.
Firstly, if
\[
\theta_2 = \left[
\begin{array}{cccc}
0 & x \tensor 1 & y \tensor 1 & 0
\end{array}
\right],
\]
then $\Ker(\theta_2) \simeq \O(0, -1) \oplus \O(-2, 0) \oplus \O(-1, 0)$ and we obtain a surjective morphism
$\F \to \O(-2, 0)$, which is absurd. The second form is
\[
\theta_2 = \left[
\begin{array}{cccc}
- 1 \tensor z & x \tensor 1 & y \tensor 1 & 0
\end{array}
\right].
\]
If $\theta_1$ is represented by a matrix having two zero columns, then $\Ker(\theta_1) \simeq 2\O(-1, -1)$, hence $\Coker(\theta_5)$
has Hilbert polynomial $2m + 2n + 2$, and hence $\Coker(\theta_5)$ is a destabilizing subsheaf of $\F$.
Thus, we may write
\[
\theta_1 = \left[
\begin{array}{ccccc}
x \tensor 1 & y \tensor 1 & 0 & 0 & 0 \\
1 \tensor z & 0 & 0 & 0 & 0 \\
0 & 1 \tensor z & 0 & 0 & 0 \\
0 & 0 & 1 \tensor z & 1 \tensor w & 0
\end{array}
\right],
\]
hence $\Ker(\theta_1) \simeq \O(-1, -2) \oplus \O(-1, -1)$, and hence $\Coker(\theta_5)$ has Hilbert polynomial $3m + 2n + 2$.
But then $\Coker(\theta_5)$ is a destabilizing subsheaf of $\F$.
We deduce that the case when $d = 1$ is also unfeasible.
\end{proof}

%%%%%%%%%%%%%%%%%%%%%%%%%%%%%%%%%%%%%%%%%%%%%%%%%%%%%%%%%%%%%

\section{Classification of sheaves}
\label{classification}

We begin our classification of semi-stable sheaves by examining the Brill-Noether locus of sheaves that do not satisfy
the first vanishing condition in Proposition \ref{vanishing}(ii).

\begin{proposition}
\label{M_2}
The sheaves $\F$ in $\MM$ satisfying the condition $\H^0(\F(0, -1)) \neq \{ 0 \}$ are precisely the non-split extension sheaves of the form
\begin{equation}
\label{C_F_p}
0 \lra \O_C(0, 1) \lra \F \lra \CC_p \lra 0,
\end{equation}
where $C \subset \PP^1 \times \PP^1$ is a curve of degree $(2, 4)$ and $p$ is a point on $C$.
Moreover, the sheaves from (\ref{C_F_p}) are precisely the sheaves $\F$ having a resolution of the form
\begin{equation}
\label{M_2_resolution}
0 \lra \O(-2, -2) \oplus \O(-1, -3) \stackrel{\f}{\lra} \O(-1, -2) \oplus \O(0, 1) \lra \F \lra 0,
\end{equation}
with $\f_{11} \neq 0$, $\f_{12} \neq 0$.
Let $\MM_2 \subset \MM$ be the subset of sheaves $\F$ from (\ref{C_F_p}).
Then $\MM_2$ is closed, irreducible, of codimension $2$, and is isomorphic to the universal curve of degree $(2, 4)$
in $\PP^1 \times \PP^1$.
Thus, $\MM_2$ is a fiber bundle with fiber $\PP^{13}$ and base $\PP^1 \times \PP^1$.
\end{proposition}

\begin{proof}
Let $\F$ give a point in $\MM$ and satisfy $\H^0(\F(0, -1)) \neq \{ 0 \}$. As in the proof of \cite[Proposition 2]{genus_two},
there is an injective morphism $\O_C \to \F(0, -1)$ for a curve $C$ of degree $(s, r)$, $0 \le s \le 2$, $0 \le r \le 4$, $1 \le r + s \le 6$.
From the stability of $\F$ we have the inequality
\[
\p(\O_C(0, 1)) = \frac{r + 2s - rs}{r + s} \le \frac{1}{6} = \p(\F),
\]
which has the unique solution $(s, r) = (2, 4)$.
We obtain extension (\ref{C_F_p}).
Conversely, let $\F$ be given by the non-split extension (\ref{C_F_p}).
As in the proof of \cite[Proposition 3]{genus_two}, we can show that $\O_C(0, 1)$ is stable,
from which it immediately follows that $\F$ gives a point in $\MM$ and that $\H^0(\F(0, -1)) \neq \{ 0 \}$.
Choose $\f_{11} \in V_1^* \tensor \CC$ and $\f_{12} \in \CC \tensor V_2^*$ defining $p$.
Since $p \in C$, we can find $\f_{21} \in \Sym^2 V_1^* \tensor \Sym^3 V_2^*$ and $\f_{22} \in V_1^* \tensor \Sym^4 V_2^*$
such that the polynomial $\f_{11} \f_{22} - \f_{12} \f_{21}$ defines $C$.
Consider the morphism
\[
\f \colon \O(-2, -2) \oplus \O(-1, -3) \lra \O(-1, -2) \oplus \O(0, 1),
\]
\[
\f = \left[
\begin{array}{cc}
\f_{11} & \f_{12} \\
\f_{21} & \f_{22}
\end{array}
\right].
\]
From the snake lemma we see that $\Coker(\f)$ is an extension of $\CC_p$ by $\O_C(0, 1)$.
Since $\Coker(\f)$ has no zero-dimensional torsion, we can apply Lemma \ref{unique_extension} to deduce that $\F \simeq \Coker(\f)$.
Thus, $[\F] \in \MM_2$ if and only if $\F$ has resolution (\ref{M_2_resolution}).
\end{proof}

In the remaining part of this section we will assume that $\F$ satisfies both vanishing conditions from Proposition \ref{vanishing}(ii).
The exact sequence (\ref{convergence}) takes the form
\begin{equation}
\label{generic_convergence}
0 \lra \Ker(\theta_1) \stackrel{\theta_5}{\lra} \O \lra \F \lra \Coker(\theta_1) \lra 0,
\end{equation}
where
\[
\theta_1 \colon 5\O(-1, -1) \lra \O(0, -1) \oplus 3\O(-1, 0).
\]

\begin{proposition}
\label{M_3_4}
Assume that $[\F] \in \MM$ and that $\H^0(\F(0, -1)) = \{ 0 \}$.
Assume that the maximal minors of $\theta_1$ have a common factor. Then $\F$ is an extension of the form
\begin{equation}
\label{Q_F_L}
0 \lra \O_Q \lra \F \lra \O_L(1, 0) \lra 0
\end{equation}
for a quintic curve $Q \subset \PP^1 \times \PP^1$ of degree $(2, 3)$ and a line $L \subset \PP^1 \times \PP^1$ of degree $(0, 1)$,
or is an extension of the form
\begin{equation}
\label{Q_p_F_L}
0 \lra \O_Q(p) \lra \F \lra \O_L \lra 0,
\end{equation}
where $\O_Q(p)$ is a non-split extension of $\CC_p$ by $\O_Q$ for a point $p \in Q$.

Conversely, any extension $\F$ as in (\ref{Q_F_L}) or (\ref{Q_p_F_L}) satisfying the condition $\H^0(\F) \simeq \CC$ is semi-stable.
Let $\MM_4 \subset \MM$ be the subset of sheaves $\F$ as in (\ref{Q_F_L}) satisfying the condition $\H^0(\F) \simeq \CC$.
Let $\MM_3 \subset \MM$ be the subset of sheaves $\F$ as in (\ref{Q_p_F_L}) satisfying the condition $\H^0(\F) \simeq \CC$.
Then $\MM_3$ and $\MM_4$ are locally closed, irreducible subsets, of codimension $3$, respectively, $4$.
\end{proposition}

\begin{proof}
Let $\eta_i$ be the maximal minor of a matrix representing $\theta_1$ obtained by deleting column $i$.
Denote $g = \gcd(\eta_1, \ldots, \eta_5)$. Let $(s, r) = (2, 4) - \deg(g)$.
It is easy to check that the sequence
\[
0 \lra \O(-s, -r) \stackrel{\eta}{\lra} 5\O(-1, -1) \stackrel{\theta_1}{\lra} \O(0, -1) \oplus 3\O(-1, 0),
\]
\[
\eta = \left[
\begin{array}{rrrrr}
\frac{\eta_1}{g} & - \frac{\eta_2}{g} & \frac{\eta_3}{g} & - \frac{\eta_4}{g} & \frac{\eta_5}{g}
\end{array}
\right]^\trans
\]
is exact.
From (\ref{generic_convergence}) we see that $\Coker(\theta_5)$ is a subsheaf of $\F$,
hence we have the inequality
\[
1 - \frac{rs}{r+s} = \p(\Coker(\theta_5)) \le \p(\F) = \frac{1}{6},
\]
forcing $(s, r) = (2, 3)$ or $(s, r) = (2, 2)$.
If $(s, r) = (2, 2)$, then $P_{\Coker(\theta_1)} = 2m + 1$ and $\Coker(\theta_1)$ is semi-stable, which follows from the semi-stability of $\F$.
But, according to \cite[Proposition 10]{ballico_huh}, $\M(2m + 1) = \emptyset$.
This contradiction shows that $(s, r) \neq (2, 2)$, hence $(s, r) = (2, 3)$.
From (\ref{generic_convergence}) we obtain the extension
\[
0 \lra \O_Q \lra \F \lra \Coker(\theta_1) \lra 0.
\]
If $\Coker(\theta_1)$ has no zero-dimensional torsion, we obtain extension (\ref{Q_F_L}).
Otherwise, the zero-dimensional torsion has length $1$, its pull-back in $\F$ is a semi-stable sheaf $\O_Q(p)$,
and we obtain extension (\ref{Q_p_F_L}).

Conversely, let $\F$ be an extension as in (\ref{Q_F_L}) satisfying $\H^0(\F) \simeq \CC$.
Assume that $\F$ had a destabilizing subsheaf $\F'$. Let $\G$ be the image of $\F'$ in $\O_L(1, 0)$.
According to \cite[Proposition 1]{genus_two}, $\O_Q$ is stable, hence $\chi(\F' \cap \O_Q) \le -1$.
Since $\chi(\F') \ge 1$, we see that $\chi(\G) \ge 2$, hence $\G = \O_L(1, 0)$ and $\O_Q \nsubseteq \F'$.
Thus $\H^0(\F' \cap \O_Q) = \{ 0 \}$, hence the map $\H^0(\F') \to \H^0(\O_L(1, 0))$ is injective.
But this map factors through $\H^0(\F) \to \H^0(\O_L(1, 0))$, which, by hypothesis, is the zero map.
We deduce that $\H^0(\F') = \{ 0 \}$, which yields a contradiction. Thus, there is no destabilizing subsheaf.
The same argument applies for extensions (\ref{Q_p_F_L}) satisfying $\H^0(\F) \simeq \CC$.

By Serre duality
\[
\Ext^1(\O_L(1, 0), \O_Q) \simeq \Ext^1(\O_Q, \O_L(-1, -2))^*.
\]
From the short exact sequence
\[
0 \lra \O(-2, -3) \lra \O \lra \O_Q \lra 0
\]
we obtain the long exact sequence
\[
\{ 0 \} = \H^0(\O_L(-1, -2)) \lra \H^0(\O_L(1, 1)) \simeq \CC^2 \lra \Ext^1(\O_Q, \O_L(-1, -2)) \lra \H^1(\O_L(-1, -2)) = \{ 0 \}.
\]
Thus $\Ext^1(\O_L(1, 0), \O_Q) \simeq \CC^2$, hence $\MM_4$ is isomorphic to an open subset of a $\PP^1$-bundle over $\PP^{11} \times \PP^1$.
By Serre duality we have
\[
\Ext^1(\O_L, \O_Q(p)) \simeq \Ext^1(\O_Q(p), \O_L(-2, -2))^*.
\]
Using Lemma \ref{unique_extension}, it is easy to see that the sheaves $\O_Q(p)$
are precisely the sheaves having a resolution of the form
\begin{equation}
\label{Q_p_resolution}
0 \lra \O(-2, -2) \oplus \O(-1, -3) \stackrel{\f}{\lra} \O(-1, -2) \oplus \O \lra \O_Q(p) \lra 0,
\end{equation}
where $\f_{11} \neq 0$, $\f_{12} \neq 0$ (cf. Proposition \ref{M_2}).
From resolution (\ref{Q_p_resolution}) we obtain the long exact sequence
\begin{align*}
\{ 0 \} = & \H^0(\O_L(-1, 0) \oplus \O_L(-2, -2)) \lra \H^0(\O_L \oplus \O_L(-1, 1)) \simeq \CC \lra \\
& \Ext^1(\O_Q(p), \O_L(-2, -2)) \lra \\
& \H^1(\O_L(-1, 0) \oplus \O_L(-2, -2)) \simeq \CC \lra \H^1(\O_L \oplus \O_L(-1, 1)) = \{ 0 \}.
\end{align*}
Thus, $\Ext^1(\O_L, \O_Q(p)) \simeq \CC^2$, hence $\MM_3$ has dimension $14$. The other claims about $\MM_3$ are obvious.
\end{proof}

\begin{lemma}
\label{generic_sheaves}
Assume that $[\F] \in \MM$ and $\H^0(\F(0, -1)) = \{ 0 \}$.
Assume that the maximal minors of $\theta_1$ have no common factor.
Then $\Ker(\theta_1) \simeq \O(-2, -4)$ and $\Coker(\theta_1) \simeq {\mathcal Ext}^2(\O_Z, \O)$ with $Z$ described below.
We have an extension
\begin{equation}
\label{generic_extension}
0 \lra \O_C \lra \F \lra {\mathcal Ext}^2(\O_Z, \O) \lra 0,
\end{equation}
where $C$ is a curve of degree $(2, 4)$ and $Z \subset C$ is a subscheme of dimension zero and length $3$.
Moreover, $Z$ is not contained in a line of degree $(0, 1)$.
\end{lemma}

\begin{proof}
The fact that $\Ker(\theta_1) \simeq \O(-2, -4)$ is well-known.
The Hilbert polynomial of $\Coker(\theta_1)$ is $3$, hence $\Coker(\theta_1)$ has dimension zero and length $3$.
From (\ref{generic_convergence}), we obtain the exact sequence
\[
0 \lra \O_C \lra \F \lra {\mathcal Coker}(\theta_1) \lra 0.
\]
We can now apply Lemma \ref{unique_extension} to obtain the extension (\ref{generic_extension}) and the isomorphism
$\Coker(\theta_1) \simeq {\mathcal Ext}^2(\O_Z, \O)$.

Assume that $Z$ is contained in a line $L$ of degree $(0, 1)$. Then $\O_Z \simeq {\mathcal Ext}^2(\O_Z, \O)$.
Choose $\f_{11} \in \CC \tensor V_2^*$ defining $L$.
Choose $\f_{12} \in \Sym^3 V_1^* \tensor \CC$ such that $\f_{11}$ and $\f_{12}$ define $Z$.
If $L \nsubseteq C$, then $L.C = 2$, which contradicts the fact that $Z \subset L \cap C$.
Thus $L \subset C$, so there is $\f_{22} \in \Sym^2 V_1^* \tensor \Sym^3 V_2^*$ such that $\f_{11} \f_{22}$ is a defining polynomial of $C$.
Consider the exact sequence
\[
0 \lra \O(1, -4) \oplus \O(-2, -3) \stackrel{\f}{\lra} \O(1, -3) \oplus \O \lra \F' \lra 0,
\]
\[
\f = \left[
\begin{array}{cc}
\f_{11} & \f_{12} \\
0 & \f_{22}
\end{array}
\right].
\]
Then $\F'$ is an extension of $\O_Z$ by $\O_C$ without zero-dimensional torsion.
Since, from the exact sequence (\ref{generic_extension}),
$\F$ is also an extension of $\O_Z$ by $\O_C$ without zero-dimensional torsion,
we can apply Lemma \ref{unique_extension} to deduce that $\F \simeq \F'$.
We obtain a contradiction from the isomorphisms $\CC\simeq \H^0(\F) \simeq \H^0(\F') \simeq \CC^3$.
\end{proof}

\begin{proposition}
\label{M_0}
Let $\MM_0 \subset \MM$ be the subset of sheaves $\F$ for which $\H^0(\F(0, -1))$ $= \{ 0 \}$, $\Ker(\theta_1) \simeq \O(-2, -4)$
and $\operatorname{supp}(\Coker(\theta_1))$ is not contained in a line of degree $(1, 0)$ or $(0, 1)$.
Then $\MM_0$ is open and can be described as the subset of sheaves $\F$ having a resolution of the form
\begin{equation}
\label{M_0_resolution}
0 \lra \O(-1, -3) \oplus \O(0, -3) \oplus \O(-1, -2) \stackrel{\f}{\lra} \O(0, -2) \oplus \O(0, -2) \oplus \O \lra \F \lra 0,
\end{equation}
where $\f_{12}$ and $\f_{22}$ are linearly independent and the maximal minors of the matrix $(\f_{ij})_{i = 1, 2, j = 1, 2, 3}$
have no common factor.
\end{proposition}

\begin{proof}
Let $\F$ give a point in $\MM_0$. Let $Z$ and $C$ be as in Lemma \ref{generic_sheaves}.
By hypothesis $Z$ is not contained in a line of degree $(1, 0)$ or $(0, 1)$,
hence ${\mathcal Ext}^2(\O_Z, \O) \simeq \Coker(\zeta)$ as in (\ref{Z_resolution}).
Let $\zeta_1$, $\zeta_2$, $\zeta_3$ be the maximal minors of $\zeta$.
They are the defining polynomials of $Z$, hence we can find
$\f_{31} \in V_1^* \tensor \Sym^3 V_2^*$, $\f_{32} \in \CC \tensor \Sym^3 V_2^*$, $\f_{33} \in V_1^* \tensor \Sym^2 V_2^*$
such that $\zeta_1 \f_{31} - \zeta_2 \f_{32} + \zeta_3 \f_{33}$ is the polynomial defining $C$.
Let
\[
\f = \left[
\begin{array}{ccc}
& \zeta \\
\f_{31} & \f_{32} & \f_{33}
\end{array}
\right].
\]
Then $\Coker(\f)$ is an extension of ${\mathcal Ext}^2(\O_Z, \O)$ by $\O_C$ without zero-dimensional torsion and,
by Lemma \ref{generic_sheaves}, the same is true of $\F$.
From Lemma \ref{unique_extension} we deduce that $\F \simeq \Coker(\f)$.
By Proposition \ref{vanishing}, $\H^0(\F) \simeq \CC$, hence the map $\H^1(\O(0, -3)) \to \H^1(2\O(0, -2))$ is injective,
which is equivalent to saying that $\f_{12}$ and $\f_{22}$ are linearly independent.
We have shown that $\F$ has resolution (\ref{M_0_resolution}).

Conversely, assume that $\F$ has resolution (\ref{M_0_resolution}).
Then $\H^0(\F) \simeq \CC$ because $\f_{12}$ and $\f_{22}$ are linearly independent.
From the snake lemma we see that $\F$ is an extension of ${\mathcal Ext}^2(\O_Z, \O)$ by $\O_C$,
where $Z$ is the zero-dimensional scheme of length $3$ given by the maximal
minors of the matrix obtained by deleting the third row of $\f$,
and $C$ is the curve of degree $(2, 4)$ defined by $\det(\f)$.
Thus, $\H^0(\F)$ generates $\O_C$.
We will show that $\F$ is semi-stable. Assume that $\F$ had a destabilizing subsheaf $\F'$.
Then $\chi(\F') > 0$ and $\chi(\F') \le \dim_{\CC}^{} \H^0(\F) = 1$, hence $\chi(\F') = 1$, forcing $\H^0(\F') \simeq \CC$.
Thus $\H^0(\F') = \H^0(\F)$, hence $\O_C \subset \F'$, and hence $\F'$ has multiplicity $6$.
There are no destabilising subsheaves of $\F$ of multiplicity $6$.
Thus, $\F$ gives a point in $\MM$.
Since $\f_{12}$ and $\f_{22}$ are linearly independent, we have $\H^0(\F(0, -1)) = \{ 0 \}$.
Since $\H^0(\F)$ generates $\O_C$, $\Ker(\theta_1) \simeq \O(-2, -4)$ and $\Coker(\theta_1) \simeq {\mathcal Ext}^2(\O_Z, \O)$.
Note that $Z$ is not contained in a line of degree $(1, 0)$ or $(0, 1)$.
In conclusion, $\F$ gives a point in $\MM_0$.
\end{proof}

\begin{proposition}
\label{rationality}
The variety $\MM$ is rational.
\end{proposition}

\begin{proof}
By Lemma \ref{length_3_scheme}, Lemma \ref{unique_extension}, Lemma \ref{generic_sheaves} and Proposition \ref{M_0},
the open subset of $\MM_0$, given by the condition that $Z$ consist of three distinct points,
is a $\PP^{11}$-bundle over an open subset of $\Hilb_{\PP^1 \times \PP^1}(3)$, so it is rational.
\end{proof}

\begin{proposition}
\label{M_2'}
Let $\F$ be an extension as in (\ref{generic_extension}) without zero-dimensional torsion, for a curve $C$ of degree $(2, 4)$
and a subscheme $Z \subset C$ that is the intersection of two curves of degree $(1, 0)$, respectively, $(0, 3)$.
Then $\F$ gives a point in $\MM$. Let $\MM_2' \subset \MM$ be the subset of such sheaves $\F$.
Then $\MM_2'$ is closed, irreducible, of codimension $2$, and can be described as the set of sheaves $\F$ having a resolution of the form
\begin{equation}
\label{M_2'_resolution}
0 \lra \O(-2, -1) \oplus \O(-1, -4) \stackrel{\f}{\lra} \O(-1, -1) \oplus \O \lra \F \lra 0,
\end{equation}
with $\f_{11} \neq 0$, $\f_{12} \neq 0$.
\end{proposition}

\begin{proof}
Note that $\O_Z \simeq {\mathcal Ext}^2(\O_Z, \O)$.
Let $\F$ be an extension of $\O_Z$ by $\O_C$ without zero-dimensional torsion.
Let $\f_{11} \in V_1^* \tensor \CC$ and $\f_{12} \in \CC \tensor \Sym^3 V_2^*$ be the defining polynomials of $Z$.
We can find $\f_{21} \in \Sym^2 V_1^* \tensor V_2^*$ and $\f_{22} \in V_1^* \tensor \Sym^4 V_2^*$ such that $\f_{11} \f_{22} - \f_{12} \f_{21}$
is the defining polynomial of $C$.
Then the cokernel of $\f = (\f_{ij})_{1 \le i, j \le 2}$ is an extension of $\O_Z$ by $\O_C$ without zero-dimensional torsion,
hence, by Lemma \ref{unique_extension}, $\F \simeq \Coker(\f)$.
Conversely, arguing as in Proposition \ref{M_0}, we can show that any sheaf of the form $\Coker(\f)$, with $\f$ as in (\ref{M_2'_resolution}),
is semi-stable.
\end{proof}

\noindent
\emph{Proof of Theorem \ref{main_theorem}}. By Propositions \ref{M_2}, \ref{M_3_4}, \ref{M_0} and \ref{M_2'},
$\MM$ is the union of the subvarieties $\MM_0$, $\MM_2$, $\MM_2'$, $\MM_3$, $\MM_4$.
For $[\F] \in \MM_2$, we have $\H^0(\F) \simeq \CC^2$, whereas, for $[\F]$ in any of the other subvarieties, we have $\H^0(\F) \simeq \CC$.
Thus, $\MM_2$ is disjoint from the other subvarieties.
For $[\F] \in \MM_0 \cup \MM_2'$, $\H^0(\F)$ generates the structure sheaf of a curve $C$ of degree $(2, 4)$,
whereas, for $[\F] \in \MM_3 \cup \MM_4$, $\H^0(\F)$ generates the structure sheaf of a curve $Q$ of degree $(2, 3)$.
Thus, $\MM_0 \cup \MM_2'$ is disjoint from $\MM_3 \cup \MM_4$.
For $[\F] \in \MM_0$, the support of $\F/\O_C$ is not contained in a line of degree $(1, 0)$,
whereas, for $[\F] \in \MM_2'$, the support of $\F/\O_C$ is contained in a line of degree $(1, 0)$.
Thus, $\MM_0$ is disjoint from $\MM_2'$.
For $[\F] \in \MM_3$, $\F/\O_Q$ has zero-dimensional torsion, whereas, for $[\F] \in \MM_4$, $\F/\O_Q$ is pure.
Thus, $\MM_3$ is disjoint from $\MM_4$.
In conclusion, the subvarieties in question form a decomposition of $\MM$.
\qed

%%%%%%%%%%%%%%%%%%%%%%%%%%%%%%%%%%%%%%%%%%%%%%%%%%%%%%%%%%%%%

\section{Variation of the moduli spaces of $\alpha$-semi-stable pairs}
\label{variation}

A \emph{coherent system} $\Lambda = (\Gamma, \F)$ on $\PP^1 \times \PP^1$
consists of a coherent algebraic sheaf $\F$ on $\PP^1 \times \PP^1$ and a vector subspace $\Gamma \subset \H^0(\F)$.
Let $\alpha$ be a positive real number and let $P_{\F}(m, n) = rm + sn + t$ be the Hilbert polynomial of $\F$.
We define the $\alpha$-\emph{slope} of $\Lambda$ as the ratio
\[
\p_{\alpha}(\Lambda) = \frac{\alpha \dim \Gamma + t}{r+ s}.
\]
We say that $\Lambda$ is $\alpha$-\emph{semi-stable}, respectively, $\alpha$-\emph{stable}, if $\F$ is pure and for any proper
coherent subsystem $\Lambda' \subset \Lambda$ we have $\p_{\alpha}(\Lambda') \le \p_{\alpha}(\Lambda)$, respectively,
$\p_{\alpha}(\Lambda') < \p_{\alpha}(\Lambda)$.
According to \cite{lepotier_asterisque} and \cite{he},
for fixed positive real number $\alpha$, non-negative integer $k$ and linear polynomial $P(m, n)$,
there is a coarse moduli space, denoted $\Syst(\PP^1 \times \PP^1, \alpha, k, P)$,
which is a projective scheme whose closed points
are in a bijective correspondence with the set of S-equivalence classes of $\alpha$-semi-stable coherent systems $(\Gamma, \F)$
on $\PP^1 \times \PP^1$ for which $\dim \Gamma = k$ and $P_{\F} = P$.
When $k = 0$ this space is $\M(P)$.
A coherent system for which $\dim \Gamma = 1$ will be called a \emph{pair}.
Our main concern is with the moduli space of $\alpha$-semi-stable pairs $\M^{\alpha}(P) = \Syst(\PP^1 \times \PP^1, \alpha, 1, P)$.
It is known that there are finitely many positive rational numbers $\alpha_1 < \ldots < \alpha_n$, called \emph{walls},
such that the set of $\alpha$-semi-stable pairs with Hilbert polynomial $P$ remains unchanged as $\alpha$ varies in one of the intervals 
$(0, \alpha_1)$, or $(\alpha_i, \alpha_{i+1})$, or $(\alpha_n, \infty)$.
In fact, from the definition of $\alpha$-semi-stability, we can see that, if $\alpha$ is a wall, then there is a strictly $\alpha$-semi-stable pair,
i.e. a pair $\Lambda$ for which there exists a subpair or quotient pair $\Lambda'$, such that $\p_{\alpha}(\Lambda) = \p_{\alpha}(\Lambda')$. This equation has only rational solutions in $\alpha$.
For $\alpha \in (\alpha_n, \infty)$ we write $\M^{\infty}(P) = \M^{\alpha}(P)$.
For $\alpha \in (0, \alpha_1)$ we write $\M^{0+}(P) = \M^{\alpha}(P)$.
If $\gcd(r+s, t) = 1$, then, from the definition of $\alpha$-semi-stability, we see that $(\Gamma, \F) \in \M^{0+}(P)$ if and only if
$\F$ is semi-stable. At the other extreme we have the following proposition due to Pandharipande and Thomas.

\begin{proposition}
\label{M_infinity}
For $\alpha \gg 0$, a pair $\Lambda = (\Gamma, \F)$ is $\alpha$-semi-stable if and only if $\F$ is pure and $\F/\O_C$ has dimension zero or is zero,
where $\O_C$ is the subsheaf of $\F$ generated by $\Gamma$. In particular, $t \ge r + s - rs$.
The scheme $\M^{\infty}(rm + sn + t)$ is isomorphic to the relative Hilbert scheme of zero-dimensional schemes of length $ t - r - s + rs$
contained in curves of degree $(s, r)$.
\end{proposition}

\begin{proof}
Assume that $(\Gamma, \F)$ is $\alpha$-semi-stable for $\alpha \gg 0$. If $P_{\O_C}(m, n) = r'm + s'n + t'$ with $r' + s' < r + s$, then
\[
\p_{\alpha}(\Gamma, \O_C) = \frac{\alpha + t'}{r' + s'} > \frac{\alpha + t}{r + s} = \p_{\alpha}(\Lambda) \quad \text{for $\alpha \gg 0$},
\]
which contradicts our hypothesis. Thus, $P_{\O_C}(m, n) = rm + sn + r + s - rs$.
Conversely, assume that $\O_C$ has this Hilbert polynomial and that $\F$ is pure.
Let $\Lambda' = (\Gamma', \F') \subset \Lambda$ be a proper coherent subsystem with $P_{\F'}(m, n) = r'm + s'n + t'$.
If $\Gamma' = \{ 0 \}$, then
\[
\p_{\alpha}(\Lambda') = \frac{t'}{r' + s'} < \frac{\alpha + t}{r + s} = \p_{\alpha}(\Lambda) \quad \text{for $\alpha \gg 0$}.
\]
If $\Gamma' = \Gamma$, then $\O_C \subset \F'$, hence $r' = r$, $s' = s$, $t' < t$, and we have
\[
\p_{\alpha}(\Lambda') = \frac{\alpha + t'}{r + s} < \frac{\alpha + t}{r + s} = \p_{\alpha}(\Lambda).
\]
The isomorphism between $\M^{\infty}(P)$ and the relative Hilbert scheme is a particular case of \cite[Proposition B.8]{pandharipande_thomas}.
As a map, it is given by $(\Gamma, \F) \mapsto (Z, C)$, where $Z \subset C$ is the subscheme introduced at Lemma \ref{unique_extension}.
\end{proof}

\begin{corollary}
\label{M_infinity_smooth}
The scheme $\M^{\infty}(4m + 2n +1)$ is isomorphic to a fiber bundle with fiber $\PP^{11}$ and base the Hilbert scheme of three points
in $\PP^1 \times \PP^1$, so it is smooth.
\end{corollary}

\begin{proof}
The relative Hilbert scheme of pairs $(Z, C)$, where $C \subset \PP^1 \times \PP^1$ is a curve of degree $(2, 4)$ and $Z \subset C$
is a subscheme of dimension zero and length $3$, has fiber $\PP(\H^0(\I_Z(2, 4)))$ over $Z$.
If $Z$ is not contained in a line of degree $(0, 1)$ or $(1, 0)$, then, from Lemma \ref{length_3_scheme}, we deduce that
$\H^0(\I_Z(2, 4)) \simeq \CC^{12}$. If $Z$ is contained in such a line, then it is straightforward to check that $\H^0(\I_Z(2, 4)) \simeq \CC^{12}$.
\end{proof}

\begin{lemma}
\label{alpha_nonempty}
Assume that $\M^{\alpha}(rm + sn + t) \neq \emptyset$. Then $t \ge r + s - rs$.
For $r$, $s$ non-negative integers, not both zero, and $\alpha \in (0, \infty)$, we have
\[
\M^{\alpha}(rm + sn + r + s - rs) \simeq \M^{\infty}(rm + sn + r + s - rs).
\]
\end{lemma}

\begin{proof}
We use induction on $r + s$. If $r + s = 1$, or if there is no wall in $[\alpha, \infty)$,
then $\M^{\alpha}(rm + sn + t) = \M^{\infty}(rm + sn + t)$ and the conclusion follows from Proposition \ref{M_infinity}.
Assume that $r + s > 1$ and that there is a wall $\alpha' \in [\alpha, \infty)$.
There is a pair $\Lambda \in \M^{\alpha'}(rm + sn + t)$ and a subpair or quotient pair $\Lambda' \in \M^{\alpha'}(r'm + s'n + t')$,
such that $\p_{\alpha'}(\Lambda) = \p_{\alpha'}(\Lambda')$. We have $0 \le r' \le r$, $0 \le s' \le s$, $1 \le r' + s' < r +s$,
\[
\frac{\alpha' + t}{r + s} = \frac{\alpha' + t'}{r' + s'},
\]
hence
\begin{align*}
t & = \frac{(r + s - r' - s') \alpha' + (r + s) t'}{r' + s'}  > \frac{r + s}{r' + s'} t' \\
& \ge \frac{r + s}{r' + s'} (r' + s' - r's') \qquad \text{(by the induction hypothesis)} \\
& = r + s - \frac{r + s}{r' + s'} r's' \ge r + s - rs.
\end{align*}
If $t = r + s - rs$, then there is no wall in $[\alpha, \infty)$, hence we have an isomorphism as in the lemma.
\end{proof}

\begin{proposition}
\label{walls}
With respect to $P(m, n) = 4m + 2n + 1$ there are only two walls at $\alpha_1 = 5$ and $\alpha_2 = 11$.
\end{proposition}

\begin{proof}
Assume that $\alpha$ is a wall. Then there are pairs $\Lambda \in \M^{\alpha}(4m + 2n + 1)$ and $\Lambda' \in \M^{\alpha}(rm + sn + t)$
such that $\Lambda'$ is a subpair or a quotient pair of $\Lambda$ and
\begin{equation}
\label{alpha}
\frac{\alpha + t}{r + s} = \frac{\alpha + 1}{6}.
\end{equation}
Here $0 \le r \le 4$, $0 \le s \le 2$, $1 \le r + s \le 5$. By Lemma \ref{alpha_nonempty}, we also have $t \ge r + s - rs$.
Assume that $r = 3$, $s = 2$, $t \ge -1$. Equation (\ref{alpha}) has solutions
$\alpha_1 = 5$ for $t = 0$ and $\alpha_2 = 11$ for $t = -1$. Assume that $r = 2$, $s = 2$, $t \ge 0$.
Equation (\ref{alpha}) has solution $\alpha = 2$ for $t = 0$.
In this case either $\Lambda \in \Ext^1(\Lambda', \Lambda'')$ or $\Lambda \in \Ext^1(\Lambda'', \Lambda')$
for some $\Lambda'' \in \M(2m + 1)$.
However, according to \cite[Proposition 10]{ballico_huh}, $\M(2m + 1) = \emptyset$.
Thus, there is no wall at $\alpha = 2$.
For all other choices of $r$ and $s$ equation (\ref{alpha}) has no positive solution in $\alpha$.
\end{proof}

\noindent
Denote $\MM^{\alpha} = \M^{\alpha}(4m + 2n + 1)$.
For $\alpha \in (11, \infty)$, write $\MM^{\alpha} = \MM^{\infty}$.
For $\alpha \in (5, 11)$, write $\MM^{\alpha} = \MM^{5+} = \MM^{11-}$.
For $\alpha \in (0, 5)$, write $\MM^{\alpha} = \MM^{0+}$.
The inclusions of sets of $\alpha$-semi-stable pairs induce the birational morphisms
\[
\xymatrix
{
\MM^{\infty} \ar[dr]_-{\rho_{\infty}} & & \MM^{11-} \ar[dl]^-{\rho_{11}} \ar@{=}[r] & \MM^{5+} \ar[dr]_-{\rho_5} & & \MM^{0+} \ar[dl]^-{\rho_0} \\
& \MM^{11} & & & \MM^5
}
\]
In view of Theorem \ref{wall_crossing}, the above are flipping diagrams
(consult \cite[Remark 5]{genus_two} for details).

\begin{remark}
\label{flipping_base}
From the proof of Proposition \ref{walls}, we see that an S-equivalence class of
strictly $\alpha$-semi-stable elements in $\MM^{11}$
consists of (split or non-split) extensions of $(\Gamma_1, \E_1)$ by $(0, \O_L(1, 0))$, 
together with the extensions of $(0, \O_L(1, 0))$ by $(\Gamma_1, \E_1)$.
Here $(\Gamma_1, \E_1)$ lies in $\M^{11}(3m + 2n -1)$ and
$L \subset \PP^1 \times \PP^1$ is a line of degree $(0, 1)$.
We say, for short, that the strictly $\alpha$-semi-stable elements of $\MM^{11}$ are of the form
$(\Gamma_1, \E_1) \oplus(0, \O_L(1, 0))$.
According to Lemma \ref{alpha_nonempty} and Proposition \ref{M_infinity},
$\E_1 \simeq \O_Q$ for a quintic curve $Q \subset \PP^1 \times \PP^1$ of degree $(2, 3)$.
Thus, $\M^{11}(3m + 2n -1) \simeq \PP^{11}$.

Again from the proof of Proposition \ref{walls}, we see that the
strictly $\alpha$-semi-stable elements in $\MM^5$ are of the form
$(\Gamma, \E) \oplus (0, \O_L)$, where $(\Gamma, \E) \in \M^5(3m + 2n)$.
We claim that $\M^5(3m + 2n) \simeq \M^{\infty}(3m + 2n)$.
To see this, we will show that there are no walls relative to the Polynomial $P(m, n) = 3m + 2n$.
As in the proof of Proposition \ref{walls}, we attempt to solve the equation
\[
\frac{\alpha + t}{r + s} = \frac{\alpha}{5}
\]
with $0 \le r \le 3$, $0 \le s \le 2$, $1 \le r + s \le 4$, $t \ge r + s - rs$.
For all choices of $r$ and $s$ we have $t \ge 0$, hence the above equation has no positive solutions in $\alpha$.
From Proposition \ref{M_infinity} we see that $\M^5(3m + 2n)$ isomorphic to the universal quintic of degree $(2, 3)$,
so it is a $\PP^{10}$-bundle over $\PP^1 \times \PP^1$.
More precisely, the elements in $\M^5(3m + 2n)$ are of the form $(\H^0(\O_Q(p)), \O_Q(p))$,
where $\O_Q(p)$ is a non-split extension of $\CC_p$ by $\O_Q$.
\end{remark}

\begin{proposition}
\label{M_3_closure}
Let $Q \subset \PP^1 \times \PP^1$ be a quintic curve of degree $(2, 3)$, let $p \in Q$ be a point, let $\O_Q(p)$ be a non-split
extension of $\CC_p$ by $\O_Q$, and let $L \subset \PP^1 \times \PP^1$ be a line of degree $(0, 1)$.
Then any non-split extension sheaf $\F$ as in (\ref{Q_p_F_L}) is semi-stable.
The set of such sheaves is the closure of $\MM_3$ in $\MM$.
The boundary $\overline{\MM}_3 \setminus \MM_3$ is contained in $\MM_2$, more precisely,
it consists of extensions as in (\ref{C_F_p}) in which $C = Q \cup L$ and $p \in Q$.
\end{proposition}

\begin{proof}
The case when $\H^0(\F) \simeq \CC$ was examined at Proposition \ref{M_3_4}, 
so we need only consider the case when $\H^0(\F) \simeq \CC^2$.
In this case the canonical morphism $\O \to \O_L$ lifts to a morphism $\O \to \F$, hence we can combine resolution (\ref{Q_p_resolution})
with the standard resolution of $\O_L$ to obtain the resolution
\[
0 \lra \O(-2, -2) \oplus \O(-1, -3) \oplus \O(0, -1) \overset{\f}{\lra} \O(-1, -2) \oplus \O \oplus \O \lra \F \lra 0,
\]
\[
\f = \left[
\begin{array}{ccc}
\f_{11} & \f_{12} & 0 \\
\f_{21} & \f_{22} & \f_{23} \\
0 & 0 & \f_{33}
\end{array}
\right],
\]
where $\f_{11} \neq 0$, $\f_{12} \neq 0$, and $\f_{23}$ and $\f_{33}$ are linearly independent.
Note that $p$ is given by the equations $\f_{11} = 0$, $\f_{12} = 0$.
From the snake lemma, we obtain an extension
\[
0 \lra \F' \lra \F \lra \CC_p \lra 0,
\]
where $\F'$ is given by the resolution
\[
0 \lra \O(-2, -3) \oplus \O(0, -1) \overset{\f'}{\lra} 2\O \lra \F' \lra 0,
\]
\[
\f' = \left[
\begin{array}{cc}
\f'_{11} & \f_{23} \\
0 & \f_{33}
\end{array}
\right], \qquad \f'_{11} = \f_{11} \f_{22} - \f_{12} \f_{21}.
\]
We claim that $\F' \simeq \O_C(0, 1)$, where $C = Q \cup L$.
In view of Proposition \ref{M_2}, the claim implies that $\F$ is semi-stable, in fact $[\F] \in \MM_2$.
It remains to prove the claim. Let ${\mathcal K}$ be the kernel of the canonical morphism $\O_C \to \O_Q$.
Since ${\mathcal K}$ has no zero-dimensional torsion and $P_{\mathcal K} = m - 1$, ${\mathcal K} \simeq \O_L( -2, 0)$.
Applying ${\mathcal Hom}(-, \omega)$ to the exact sequence
\[
0 \lra \O_L(-2, 0) \lra \O_C(0, 1) \lra \O_Q(0, 1) \lra 0,
\]
yields the exact sequence
\[
0 \lra {\mathcal Ext}^1(\O_Q(0, 1), \omega) \lra {\mathcal Ext}^1(\O_C(0, 1), \omega) \lra {\mathcal Ext}^1(\O_L(-2, 0), \omega) \lra 0,
\]
which is the same as the exact sequence
\[
0 \lra \O_Q \lra \O_C(0, 1) \lra \O_L \lra 0.
\]
Since $\H^0(\O_C(0, 1)) \simeq \CC^2$, the canonical morphism $\O \to \O_L$ lifts to a morphism $\O \to \O_C(0, 1)$,
hence the canonical resolutions of $\O_Q$ and $\O_L$ can be combined into a resolution of the form
\[
0 \lra \O(-2, -3) \oplus \O(0, -1) \overset{\psi}{\lra} 2\O \lra \O_C(0, 1) \lra 0,
\]
\[
\psi = \left[
\begin{array}{cc}
\f'_{11} & \psi_{12} \\
0 & \f_{33}
\end{array}
\right].
\]
Since $\O_C(0, 1)$ is a non-split extension of $\O_L$ by $\O_Q$, $\psi_{12}$ and $\f_{33}$ are linearly independent.
It is clear now that the matrices representing $\f'$ and $\psi$ are equivalent under elementary row and column operations.
We conclude that $\F' \simeq \O_C(0, 1)$.
\end{proof}

\noindent
The preimages of the sets of strictly semi-stable elements are the flipping loci:
\begin{align*}
F^{\infty} & = \rho_{\infty}^{-1}(\M^{11}(3m + 2n - 1) \times \M(m+ 2)) \subset \MM^{\infty}, \\
F^{11} & = \rho_{11}^{-1}(\M^{11}(3m + 2n - 1) \times \M(m+ 2)) \subset \MM^{11-}, \\
F^5 & = \rho_5^{-1}(\M^5(3m + 2n) \times \M(m + 1)) \subset \MM^{5+}, \\
F^0 & = \rho_0^{-1}(\M^5(3m + 2n) \times \M(m + 1)) \subset \MM^{0+}.
\end{align*}

\begin{proposition}
\label{flipping_loci}
Consider $\Lambda_1 \in \M^{11}(3m + 2n - 1)$, $\Lambda_2 \in \M(m + 2)$, $\Lambda_3 \in \M^5(3m + 2n)$,
and $\Lambda_4 \in \M(m + 1)$.
\begin{enumerate}
\item[(i)] Over a point $(\Lambda_1, \Lambda_2)$, $F^{\infty}$ has fiber $\PP(\Ext^1(\Lambda_1, \Lambda_2))$.
\item[(ii)] Over a point $(\Lambda_1, \Lambda_2)$, $F^{11}$ has fiber $\PP(\Ext^1(\Lambda_2, \Lambda_1))$.
\item[(iii)] Over a point $(\Lambda_3, \Lambda_4)$, $F^5$ has fiber $\PP(\Ext^1(\Lambda_3, \Lambda_4))$.
\item[(iv)] Over a point $(\Lambda_3, \Lambda_4)$, $F^0$ has fiber $\PP(\Ext^1(\Lambda_4, \Lambda_3))$.
\end{enumerate}
\end{proposition}

\begin{proof}
(i) We refer to the argument at \cite[Remark 2]{genus_two}.

\medskip

\noindent
(ii) Assume that $\Lambda = (\Gamma, \F) \in F^{11}$ lies over $(\Lambda_1, \Lambda_2)$.
Then $\Lambda$ is a non-split extension of $\Lambda_1$ by $\Lambda_2$, or, vice versa, of $\Lambda_2$ by $\Lambda_1$.
If $\Lambda_2 \subset \Lambda$, then
\[
\p_{\alpha}(\Lambda_2) = 2 > \frac{\alpha + 1}{6} = \p_{\alpha}(\Lambda) \quad \text{for $\alpha \in (5, 11)$},
\]
which violates the semi-stability of $\Lambda$. Thus $\Lambda \in \PP(\Ext^1(\Lambda_2, \Lambda_1))$.
Conversely, given such $\Lambda$, we need to show that $\Lambda \in \MM^{\alpha}$ for $\alpha \in (5, 11)$.
Write $\Lambda_1 = (\Gamma_1, \O_Q)$, $\Lambda_2 = (0, \O_L(1, 0))$.
We have a non-split extension of sheaves
\[
0 \lra \O_Q \lra \F \lra \O_L(1, 0) \lra 0.
\]
Let $\Lambda' = (\Gamma', \F')$ be a proper coherent subsystem of $\Lambda$.
Let $\G$ be the image of $\F'$ in $\O_L(1, 0)$.
If $\F' \cap \O_Q = \{ 0 \}$, then $\G \neq \O_L(1, 0)$, forcing $\chi(\F') = \chi(\G) \le 1$.
If $\F' \cap \O_Q \neq \{ 0 \}$, then $\chi(\F' \cap \O_Q) \le -1$ because, by virtue of \cite[Lemma 9]{ballico_huh}, $\O_Q$ is semi-stable.
We have in this case $\chi(\F') = \chi(\F' \cap \O_Q) + \chi(\G) \le -1 + 2 = 1$.
If $\Gamma' = \{ 0 \}$, then
\[
\p_{\alpha}(\Lambda') = \p(\F') \le 1 < \frac{\alpha + 1}{6} = \p_{\alpha}(\Lambda) \quad \text{for $\alpha \in (5, 11)$}.
\]
Assume now that $\Gamma' \neq \{ 0 \}$. Then $\Gamma' = \Gamma = \H^0(\O_Q)$, hence $\O_Q \subset \F'$.
If $\O_Q = \F'$, then
\[
\p_{\alpha}(\Lambda') = \frac{\alpha - 1}{5} < \frac{\alpha + 1}{6} = \p_{\alpha}(\Lambda) \quad \text{for $\alpha \in (5, 11)$}.
\]
If $\O_Q \subsetneqq \F'$, then $r(\F') + s(\F') = 6$, hence $\chi(\F') \le 0$, and hence
\[
\p_{\alpha}(\Lambda') = \frac{\alpha + \chi(\F')}{6} \le \frac{\alpha}{6} < \frac{\alpha + 1}{6} = \p_{\alpha}(\Lambda).
\]
In all cases we have the inequality $\p_{\alpha}(\Lambda') < \p_{\alpha}(\Lambda)$, hence $\Lambda \in \MM^{\alpha}$, for $\alpha \in (5, 11)$.

\medskip

\noindent
(iii) We will show that every $\Lambda = (\Gamma, \F) \in \PP(\Ext^1(\Lambda_3, \Lambda_4))$ gives a point in $\MM^{\alpha}$
for $\alpha \in (5, 11)$. Write $\Lambda_3 = (\Gamma_3, \O_Q(p))$, $\Lambda_4 = (0, \O_L)$.
We have a, possibly split, extension of sheaves
\[
0 \lra \O_L \lra \F \lra \O_Q(p) \lra 0.
\]
Let $\Lambda' = (\Gamma', \F')$ be a proper coherent subsystem of $\Lambda$.
Let $\G$ be the image of $\F'$ in $\O_Q(p)$. Using the fact that $\O_Q$ is semi-stable, it is easy to see that $\O_Q(p)$ is semi-stable, as well.
Thus, $\chi(\G) \le 0$, hence $\chi(\F') = \chi(\F' \cap \O_L) + \chi(\G) \le 1 + 0 = 1$.
If $\Gamma' = \{ 0 \}$, then
\[
\p_{\alpha}(\Lambda') = \p(\F') \le 1 < \frac{\alpha + 1}{6} = \p_{\alpha}(\Lambda) \quad \text{for $\alpha \in (5, 11)$}.
\]
Assume now that $\Gamma' \neq \{ 0 \}$, i.e. $\Gamma' = \Gamma$. Then $\O_Q \subset \G$.
If $\F' \cap \O_L = \{ 0 \}$, then $\F' \ncong \O_Q(p)$, otherwise $\Lambda \simeq \Lambda_3 \oplus \Lambda_4$.
In this case $\F' \simeq \O_Q$, hence
\[
\p_{\alpha}(\Lambda') = \frac{\alpha - 1}{5} < \frac{\alpha + 1}{6} = \p_{\alpha}(\Lambda) \quad \text{for $\alpha \in (5, 11)$}.
\]
Assume now that $\F' \cap \O_L \neq \{ 0 \}$. Then $r(\F') + s(\F') = 6$, hence $\chi(\F') \le 0$, and hence
$\p_{\alpha}(\Lambda') < \p_{\alpha}(\Lambda)$.

\medskip

\noindent
(iv) If $(\Gamma, \F) \in \PP(\Ext^1(\Lambda_4, \Lambda_3))$, then we have the non-split extension (\ref{Q_p_F_L}),
hence, by Proposition \ref{M_3_closure}, $\F$ is semi-stable.
Thus $(\Gamma, \F) \in \MM^{0+}$, i.e. $(\Gamma, \F) \in F^0$.
\end{proof}

\begin{proposition}
\label{ext_sequence}
\emph{(\cite[Corollaire 1.6]{he})}
Let $\Lambda = (\Gamma, \F)$ and $\Lambda' = (\Gamma', \F')$ be two coherent systems on a separated scheme of finite type over $\CC$.
Then there is a long exact sequence
\begin{align*}
0 & \lra \Hom(\Lambda, \Lambda') \lra \Hom(\F, \F') \lra \Hom(\Gamma, \H^0(\F')/\Gamma') \\
& \lra \Ext^1(\Lambda, \Lambda') \lra \Ext^1(\F, \F') \lra \Hom(\Gamma, \H^1(\F')) \\
& \lra \Ext^2(\Lambda, \Lambda') \lra \Ext^2(\F, \F') \lra \Hom(\Gamma, \H^2(\F')).
\end{align*}
\end{proposition}

\begin{proposition}
\label{flipping_bundles}
The flipping loci $F^{\infty}$, $F^{11}$, $F^5$, $F^0$ are smooth bundles with fibers $\PP^3$, $\PP^1$, $\PP^2$, respectively, $\PP^1$.
\end{proposition}

\begin{proof}
We need to determine the extension spaces of pairs occurring at Proposition \ref{flipping_loci}.

\medskip

\noindent
(i) Choose $\Lambda_1 = (\Gamma_1, \O_Q)$ and $\Lambda_2 = (0, \O_L(1, 0))$.
From Proposition \ref{ext_sequence} we have the long exact sequence
\begin{align*}
0 & \lra \Hom(\Lambda_1, \Lambda_2) \lra \Hom(\O_Q, \O_L(1, 0)) \lra \Hom(\Gamma_1, \H^0(\O_L(1, 0))) \simeq \CC^2 \\
& \lra \Ext^1(\Lambda_1, \Lambda_2) \lra \Ext^1(\O_Q, \O_L(1, 0)) \lra \Hom(\Gamma_1, \H^1(\O_L(1, 0))) = \{ 0 \}.
\end{align*}
For $\alpha \gg 0$, $\Lambda_1$ and $\Lambda_2$ are $\alpha$-stable coherent systems of different slopes, hence
$\Hom(\Lambda_1, \Lambda_2) = \{ 0 \}$.
From the short exact sequence
\begin{equation}
\label{Q_resolution}
0 \lra \O(-2, -3) \lra \O \lra \O_Q \lra 0,
\end{equation}
we obtain the long exact sequence
\begin{align*}
0 & \lra \Hom(\O_Q, \O_L(1, 0)) \lra \H^0(\O_L(1, 0)) \simeq \CC^2 \lra \H^0(\O_L(3, 3)) \simeq \CC^4 \\
& \lra \Ext^1(\O_Q, \O_L(1, 0)) \lra \H^1(\O_L(1, 0)) = \{ 0 \}.
\end{align*}
Combining the last two long exact sequences, we obtain the isomorphism $\Ext^1(\Lambda_1, \Lambda_2) \simeq \CC^4$.

\medskip

\noindent
(ii) From Proposition \ref{ext_sequence}, we have the exact sequence
\begin{align*}
\{ 0 \} = & \Hom(0, \H^0(\O_Q)/\Gamma_1) \lra \Ext^1(\Lambda_2, \Lambda_1) \lra \\
& \Ext^1(\O_L(1, 0), \O_Q) \simeq \Ext^1(\O_Q, \O_L(-1, -2))^* \lra \Hom(0, \H^1(\O_Q)) = \{ 0 \}.
\end{align*}
From resolution (\ref{Q_resolution}), we obtain the exact sequence
\[
\{ 0 \} = \H^0(\O_L(-1, -2)) \lra \H^0(\O_L(1, 1)) \simeq \CC^2 \lra \Ext^1(\O_Q, \O_L(-1, -2)) \lra \H^1(\O_L(-1, -2)) = \{ 0 \}.
\]
Combining the last two exact sequences, we obtain the isomorphism $\Ext^1(\Lambda_2, \Lambda_1) \simeq \CC^2$.

\medskip

\noindent
(iii) Choose $\Lambda_3 = (\Gamma, \O_Q(p))$ and $\Lambda_4 = (0, \O_L)$.
From Proposition \ref{ext_sequence}, we have the long exact sequence
\begin{align*}
\{ 0 \} = & \Hom(\Lambda_3, \Lambda_4) \lra \Hom(\O_Q(p), \O_L) \lra \Hom(\Gamma, \H^0(\O_L)) \simeq \CC \lra \\
& \Ext^1(\Lambda_3, \Lambda_4) \lra \Ext^1(\O_Q(p), \O_L) \lra \Hom(\Gamma, \H^1(\O_L)) = \{ 0 \}.
\end{align*}
From resolution (\ref{Q_p_resolution}) we obtain the exact sequence
\begin{align*}
0 & \lra \Hom(\O_Q(p), \O_L) \lra \H^0(\O_L(1, 2) \oplus \O_L) \simeq \CC^3 \lra \H^0(\O_L(2, 2) \oplus \O_L(1, 3)) \simeq \CC^5 \\
& \lra \Ext^1(\O_Q(p), \O_L) \lra \H^1(\O_L(1, 2) \oplus \O_L) = \{ 0 \}.
\end{align*}
Combining the last two exact sequences, it follows that $\Ext^1(\Lambda_3, \Lambda_4) \simeq \CC^3$.

\medskip

\noindent
(iv) From Proposition \ref{ext_sequence}, we obtain the exact sequence
\begin{align*}
\{ 0 \} = & \Hom(0, \H^0(\O_Q(p))/\Gamma) \lra \Ext^1(\Lambda_4, \Lambda_3) \lra \\
& \Ext^1(\O_L, \O_Q(p)) \simeq \Ext^1(\O_Q(p), \O_L(-2, -2))^* \lra \Hom(0, \H^1(\O_Q(p))) = \{ 0 \}.
\end{align*}
From resolution (\ref{Q_p_resolution}) we obtain the exact sequence
\begin{align*}
\{ 0 \} = & \H^0(\O_L(-1, 0) \oplus \O_L(-2, -2)) \lra \H^0(\O_L \oplus \O_L(-1, 1)) \simeq \CC \lra \Ext^1(\O_Q(p), \O_L(-2, -2)) \\
\lra & \H^1(\O_L(-1, 0) \oplus \O_L(-2, -2)) \simeq \CC \lra \H^1(\O_L \oplus \O_L(-1, 1)) = \{ 0 \}.
\end{align*}
Combining the last two exact sequences, it follows that $\Ext^1(\Lambda_4, \Lambda_3) \simeq \CC^2$.
\end{proof}

\begin{lemma}
\label{ext^2}
\emph{(i)} For $\Lambda \in F^{11}$ we have $\Ext^2(\Lambda, \Lambda) = \{ 0 \}$. \\
\emph{(ii)} For $\Lambda \in F^0$ we have $\Ext^2(\Lambda, \Lambda) = \{ 0 \}$.
\end{lemma}

\begin{proof}
(i) In view of the exact sequence
\[
0 \lra \Lambda_1 \lra \Lambda \lra \Lambda_2 \lra 0
\]
it is enough to show that $\Ext^2(\Lambda_i, \Lambda_j) = \{ 0 \}$ for $i, j = 1, 2$.
From Proposition \ref{ext_sequence}, we have the exact sequence
\[
\{ 0 \} = \Hom(\Gamma_1, \H^1(\O_L(1, 0))) \lra \Ext^2(\Lambda_1, \Lambda_2)
\lra \Ext^2(\O_Q, \O_L(1, 0)) \simeq \Hom(\O_L(1, 0), \O_Q(-2, -2))^*.
\]
The group on the right vanishes because $\H^0(\O_Q(-3, -2)) = \{ 0 \}$.
Thus, $\Ext^2(\Lambda_1, \Lambda_2)$ $= \{ 0 \}$.
From the exact sequence
\[
\{ 0 \} = \Hom(0, \H^1(\O_Q)) \lra \Ext^2(\Lambda_2, \Lambda_1)
\lra \Ext^2(\O_L(1, 0), \O_Q) \simeq \Hom(\O_Q, \O_L(-1, -2))^* = \{ 0 \}
\]
we obtain the vanishing of $\Ext^2(\Lambda_2, \Lambda_1)$.
From the exact sequence
\begin{multline*}
\{ 0 \} = \Hom(0, \H^1(\O_L(1, 0))) \lra \Ext^2(\Lambda_2, \Lambda_2) \\
\lra \Ext^2(\O_L(1, 0), \O_L(1, 0)) \simeq \Hom(\O_L(1, 0), \O_L(-1, -2))^* = \{ 0 \}
\end{multline*}
we obtain the vanishing of $\Ext^2(\Lambda_2, \Lambda_2)$.
From Proposition \ref{ext_sequence}, we have the exact sequence
\begin{align*}
\{ 0 \} = \Hom(\Gamma_1, \H^0(\O_Q)/\Gamma_1)
& \lra \Ext^1(\Lambda_1, \Lambda_1) \lra \Ext^1(\O_Q, \O_Q) \lra \Hom(\Gamma_1, \H^1(\O_Q)) \simeq \CC^2 \\
& \lra \Ext^2(\Lambda_1, \Lambda_1) \lra \Ext^2(\O_Q, \O_Q) \simeq \Hom(\O_Q, \O_Q(-2, -2))^* = \{ 0 \}.
\end{align*}
According to \cite[Th\'eor\`eme 3.12]{he}, $\Ext^1(\Lambda_1, \Lambda_1)$ is isomorphic to the tangent space of $\M^{11}(3m + 2n -1) \simeq \PP^{11}$
(see Remark \ref{flipping_base}) at $\Lambda_1$, so it is isomorphic to $\CC^{11}$.
From resolution (\ref{Q_resolution}), we obtain the exact sequence
\begin{align*}
0 \lra & \Hom(\O_Q, \O_Q) \stackrel{\simeq}{\lra} \H^0(\O_Q) \lra \H^0(\O_Q(2, 3)) \simeq \CC^{11} \\
\lra & \Ext^1(\O_Q, \O_Q) \lra \H^1(\O_Q) \simeq \CC^2 \lra \H^1(\O_Q(2, 3)) = \{ 0 \}.
\end{align*}
Combining the last two exact sequences we obtain the vanishing of $\Ext^2(\Lambda_1, \Lambda_1)$.

\medskip

\noindent
(ii) As above, we need to prove that $\Ext^2(\Lambda_i, \Lambda_j) = \{ 0 \}$ for $i, j = 3, 4$.
From Proposition \ref{ext_sequence}, we have the exact sequence
\[
\{ 0 \} = \Hom(\Gamma, \H^1(\O_L)) \lra \Ext^2(\Lambda_3, \Lambda_4) \lra \Ext^2(\O_Q(p), \O_L) \simeq \Hom(\O_L, \O_Q(p)(-2, -2))^* = \{ 0 \}.
\]
Thus, $\Ext^2(\Lambda_3, \Lambda_4) = \{ 0 \}$.
From the exact sequence
\[
\{ 0 \} = \Hom(0, \H^1(\O_Q(p))) \lra \Ext^2(\Lambda_4, \Lambda_3) \lra \Ext^2(\O_L, \O_Q(p)) \simeq \Hom(\O_Q(p), \O_L(-2, -2))^* = \{ 0 \}
\]
we obtain the vanishing of $\Ext^2(\Lambda_4, \Lambda_3)$.
From the exact sequence
\[
\{ 0 \} = \Hom(0, \H^1(\O_L)) \lra \Ext^2(\Lambda_4, \Lambda_4) \lra \Ext^2(\O_L, \O_L) \simeq \Hom(\O_L, \O_L(-2, -2))^* = \{ 0 \}
\]
we obtain the vanishing of $\Ext^2(\Lambda_4, \Lambda_4)$.
From Proposition \ref{ext_sequence}, we have the exact sequence
\begin{align*}
\{ 0 \} =  & \Hom(\Gamma, \H^0(\O_Q(p))/\Gamma) \\
\lra & \Ext^1(\Lambda_3, \Lambda_3) \lra \Ext^1(\O_Q(p), \O_Q(p)) \lra \Hom(\Gamma, \H^1(\O_Q(p))) \simeq \CC \\
\lra & \Ext^2(\Lambda_3, \Lambda_3) \lra \Ext^2(\O_Q(p), \O_Q(p)) \simeq \Hom(\O_Q(p), \O_Q(p)(-2, -2))^* = \{ 0 \}.
\end{align*}
From resolution (\ref{Q_p_resolution}), we obtain the exact sequence
\begin{align*}
0 \lra & \Hom(\O_Q(p), \O_Q(p)) \lra \H^0(\O_Q(p)(1, 2)) \oplus \H^0(\O_Q(p)) \lra \H^0(\O_Q(p)(2, 2)) \oplus \H^0(\O_Q(p)(1, 3)) \\
\lra & \Ext^1(\O_Q(p), \O_Q(p)) \lra \H^1(\O_Q(p)(1, 2)) \oplus \H^1(\O_Q(p)) \lra \H^1(\O_Q(p)(2, 2)) \oplus \H^1(\O_Q(p)(1, 3)) \lra 0.
\end{align*}
Since $\Hom(\O_Q(p), \O_Q(p)) \simeq \CC$, it follows that
\[
\dim^{}_{\CC} \Ext^1(\O_Q(p), \O_Q(p)) = 1 - \chi(\O_Q(p)(1, 2)) - \chi(\O_Q(p)) + \chi(\O_Q(p)(2, 2)) + \chi(\O_Q(p)(1, 3)) = 13.
\]
According to \cite[Th\'eor\`eme 3.12]{he}, $\Ext^1(\Lambda_3, \Lambda_3)$ is isomorphic to the tangent space at $\Lambda_3$ of $\M^5(3m+2n)$,
which, according to Remark \ref{flipping_base}, is smooth of dimension $12$.
We obtain the vanishing of $\Ext^2(\Lambda_3, \Lambda_3)$.
\end{proof}

\begin{theorem}
\label{wall_crossing}
Let $\MM^{\alpha}$ be the moduli space of $\alpha$-semi-stable pairs on $\PP^1 \times \PP^1$ with Hilbert polynomial
$P(m, n) = 4m + 2n + 1$. We have the following blowing up diagrams
\[
\xymatrix
{
& \ \ \, \widetilde{\MM}^{\infty} \ar[dl]_-{\beta_{\infty}} \ar[dr]^-{\beta_{11}} & & & \ \ \ \widetilde{\MM}^{5+} \ar[dl]_-{\beta_5} \ar[dr]^-{\beta_0} \\
\MM^{\infty} \ar[dr]_-{\rho_{\infty}} & & \MM^{11-} \ar[dl]^-{\rho_{11}} \ar@{=}[r] & \MM^{5+} \ar[dr]_-{\rho_5} & & \MM^{0+} \ar[dl]^-{\rho_0} \\
& \MM^{11} & & & \MM^5
}
\]
Here $\beta_{\infty}$ is the blow-up along $F^{\infty}$
and $\beta_{11}$ is the contraction of the exceptional divisor $\widetilde{F}^{\infty}$ in the direction of $\PP^3$,
where we view $\widetilde{F}^{\infty}$ as a $\PP^3 \times \PP^1$-bundle with base $\M^{11}(3m + 2n - 1) \times \M(m + 2)$.
Likewise, $\beta_5$ is the blow-up along $F^5$ and $\beta_0$ is the contraction of the exceptional divisor
$\widetilde{F}^5$ in the direction of $\PP^2$, where we view $\widetilde{F}^5$ as a $\PP^2 \times \PP^1$-bundle over
$\M^5(3m + 2n) \times \M(m + 1)$.
\end{theorem}

\begin{proof}
A birational morphism $\beta_{11} \colon \widetilde{\MM}^{\infty} \to \MM^{11-}$ can be constructed as at \cite[Theorem 3.3]{choi_chung}
such that $\beta_{11}$ contracts $\widetilde{F}^{\infty}$ in the direction of $\PP^3$, $\beta_{11}$ is an isomorphism outside $F^{11}$,
and $\beta_{11}^{-1}(x) \simeq \PP^3$ for any $x \in F^{11}$.
We now apply the Universal Property of the blow-up \cite[p. 604]{griffiths_harris} to deduce that $\beta_{11}$ is a blow-up with center $F^{11}$.
For this we need to know that $\MM^{11-}$ and $F^{11}$ are smooth.
By Corollary \ref{M_infinity_smooth}, $\MM^{\infty}$ is smooth,
by Proposition \ref{flipping_bundles}, the blowing up center $F^{\infty}$ is smooth, hence $\widetilde{\MM}^{\infty}$ is smooth, too.
Since $\beta_{11}$ is an isomorphism outside $F^{11}$, $\MM^{11-} \setminus F^{11}$ is smooth.
Since all points of $\MM^{11-}$ are $\alpha$-stable, we can apply the Smoothness Criterion \cite[Th\'eor\`eme 3.12]{he},
which states that $\Lambda \in \MM^{11-}$ is a smooth point if $\Ext^2(\Lambda, \Lambda) = \{ 0 \}$.
Thus, in view of Lemma \ref{ext^2}(i), $\MM^{11-}$ is smooth at every point of $F^{11}$.
The smoothness of $F^{11}$ was proved at Proposition \ref{flipping_bundles}.

For the second blow-up diagram we reason analogously, using the facts that $F^5$ and $F^0$ are smooth, and using Lemma \ref{ext^2}(ii).
\end{proof}

\noindent
According to \cite[Th\'eor\`eme 4.3]{he}, there is a universal family $(\widetilde{\Gamma}, \widetilde{\F})$
of coherent systems on $\MM^{0+} \times \PP^1 \times \PP^1$.
In particular, $\widetilde{\F}$ is a family of semi-stable sheaves on $\PP^1 \times \PP^1$ with Hilbert polynomial $4m + 2n + 1$,
which is flat over $\MM^{0+}$. It induces the so called \emph{forgetful morphism} $\phi \colon \MM^{0+} \to \MM$.
We have $\phi(\Gamma, \F) = [\F]$.

\begin{proposition}
\label{blow_up}
The forgetful morphism $\phi \colon \MM^{0+} \to \MM$ is a blow-up with center the Brill-Noether locus $\MM_2$.
\end{proposition}

\begin{proof}
According to Proposition \ref{vanishing}(ii), for $[\F] \in \MM \setminus \MM_2$ we have $\H^0(\F) \simeq \CC$,
hence $\phi^{-1}([\F]) = (\H^0(\F), \F)$ is a single point. Thus, $\phi$ is an isomorphism away from $\MM_2$.
According to Proposition \ref{M_2}, for $[\F] \in \MM_2$ we have $\H^0(\F) \simeq \CC^2$, hence $\phi^{-1}([\F]) \simeq \PP^1$.
Taking into account that $\MM$ and $\MM_2$ are smooth, we can apply the Universal Property of the blow-up \cite[p. 604]{griffiths_harris}
to conclude that $\phi$ is a blow-up with center $\MM_2$.
\end{proof}

\noindent \\
\emph{Proof of Theorem \ref{poincare_polynomial}}.
By virtue of Proposition \ref{blow_up}, we have the relation
\[
\Poly(\MM) = \Poly(\MM^{0+}) - \xi \Poly(\MM_2).
\]
According to Proposition \ref{M_2}, we have the relation
\[
\Poly(\MM_2) = \Poly(\PP^{13}) \Poly(\PP^1 \times \PP^1).
\]
By virtue of Theorem \ref{wall_crossing}, we have the relation
\begin{align*}
\Poly(\MM^{0+}) = \Poly(\MM^{\infty}) + & \big(\Poly(\PP^1) - \Poly(\PP^3)\big) \Poly\!\big(\M^{11}(3m + 2n - 1) \times \M(m + 2)\big) \\
+ & \big(\Poly(\PP^1) - \Poly(\PP^2)\big) \Poly\!\big(\M^5(3m + 2n) \times \M(m + 1)\big).
\end{align*}
In view of Corollary \ref{M_infinity_smooth} and Remark \ref{flipping_base}, we have the relation
\[
\Poly(\MM^{0+}) = \Poly(\PP^{11}) \Poly(\Hilb_{\PP^1 \times \PP^1}(3)) + (\Poly(\PP^1) - \Poly(\PP^3)) \Poly(\PP^{11}) \Poly (\PP^1)
+ (\Poly(\PP^1) - \Poly(\PP^2)) \Poly(\PP^{10}) \Poly(\PP^1 \times \PP^1) \Poly(\PP^1).
\]
According to \cite[Theorem 0.1]{goettsche}, we have the equation
\[
\Poly(\Hilb_{\PP^1 \times \PP^1}(3)) = \xi^6 + 3 \xi^5 + 9 \xi^4 + 14 \xi^3 + 9 \xi^2 + 3 \xi + 1.
\]
The final result reads
\begin{multline*}
\Poly(\MM) = \frac{\xi^{12} - 1}{\xi - 1} (\xi^6 + 3 \xi^5 + 9 \xi^4 + 14 \xi^3 + 9 \xi^2 + 3 \xi + 1) - (\xi^3 + \xi^2) \frac{\xi^{12} - 1}{\xi - 1} (\xi + 1) \\
- \xi^2 \frac{\xi^{11} - 1}{\xi - 1} (\xi + 1)^3 - \xi \frac{\xi^{14} - 1}{\xi - 1} (\xi + 1)^2. \qed
\end{multline*}

\noindent \\
{\bf Acknowledgement.}
The author would like to thank Jean-Marc Dr\'ezet for several helpful discussions.

%%%%%%%%%%%%%%%%%%%%%%%%%%%%%%%%%%%%%%%%%%%%%%%%%%%%%%%%%%%%%%%%

\end{document}